\documentclass[12pt]{article}
\usepackage{mathrsfs}
\usepackage{amsfonts}
\usepackage[all]{xy}
\usepackage{amsmath,amssymb}
\usepackage{amscd}

\baselineskip 5.5pt \lineskip 5.5pt \numberwithin{equation}{section}

\def\qed{\hfill$\Box$\par}

\def \C{\hbox{$C\hskip -5pt \vrule height 6pt depth 0pt \hskip 6pt$}}

\def\qed{\ \ \ifhmode\unskip\nobreak\fi\ifmmode\ifinner
         \else\hskip5pt\fi\fi
 \hbox{\hskip5pt\vrule width4pt height6pt depth1.5pt\hskip 1 pt}}

\def\cl{\centerline}

\def\vs{\vspace*}

\def\C{\mathbb{C}}

\def\mS{\mathfrak{S}}
\def\mH{\mathcal{H}}
\def\U{\mathcal{U}}

\newtheorem{theo}{{Theorem}}[section]
\newtheorem{lemm}[theo]{Lemma}
\newtheorem{rema}[theo]{Remark}
\newtheorem{defi}[theo]{Definition}
\newtheorem{coro}[theo]{Corollary}
\newtheorem{prop}[theo]{Proposition}

\allowdisplaybreaks
\begin{document}
\cl{\large{\bf Quasi-Whittaker modules for the Schr\"{o}dinger algebra }
\noindent\footnote{Supported by the National Science Foundation of
China (No. 11047030, 11101388, 11171055, 11001046). }}
\vspace{10pt}
\cl{Yan-an Cai}
\cl{\small School of Mathematical Sciences, University of Science }
\cl{\small and Technology of China, Hefei 230026, China}
\cl{\small E-mail: yatsai@mail.ustc.edu.cn}
\vspace{10pt}
\cl{Yongsheng Cheng}
\cl{\small School
of Mathematics and Information Sciences}
\cl{\small \& Institute of Contemporary Mathematics, } \cl{\small Henan
University, Kaifeng 475004, China}

\cl{\small E-mail:  yscheng.math@gmail.com}
\vspace{10pt}
\cl{Ran Shen}
\cl{\small College of Science, Donghua University, Shanghai, 201620, China}
\cl{\small Email:rshen@dhu.edu.cn}
\vspace{10pt}
\noindent{{\bf Abstract.}
In this paper, we construct a new class of
modules for the Schr\"{o}dinger algebra $\mS$, called quasi-Whittaker module.
Different from \cite{[ZC]}, the quasi-Whittaker
module is not induced by the Borel subalgebra of the Schr\"{o}dinger algebra
related with the triangular decomposition, but its Heisenberg subalgebra $\mH$.
We prove that, for a simple $\mS$-module $V$,
$V$ is a quasi-Whittaker module if and only if $V$ is a locally finite $\mH$-module; Furthermore, we classify the simple
quasi-Whittaker modules by the elements with the action
 similar to the center elements in $U(\mS)$
and their quasi-Whittaker vectors.
Finally, we characterize arbitrary quasi-Whittaker modules. }


\vspace{10pt}
\vs{6pt}
\par
\cl{\bf\S1. \ Introduction}
\setcounter{section}{1}\setcounter{theo}{0}\setcounter{equation}{0}
\par
In \cite{[K]}, B. Kostant defined a class of
modules for a finite-dimensional complex
semisimple Lie algebra. He called these modules as
Whittaker modules because of their
connections with the Whittaker equations that arise in the study of the
corresponding representations of the associated Lie group.
The traditional definition of Whittaker modules is closely tied to
the triangular decomposition
of a Lie algebra. Results for the complex semisimple Lie
algebras have been extended to quantum groups for $\U_q(\mathfrak{g})$
\cite{[S]}, and $\U_q(\mathfrak{sl}_2)$ \cite{[O]}, Virasoro algebra
\cite{[GL], [LGZ], [OW]}, Schr\"{o}dinger-Witt algebra \cite{[ZTL]},
Heisenberg algebra and affine Kac-Moody algebra \cite{[C], [GZ]},
Heisenberg-Virasoro algebra \cite{[CG], [LWZ]},
Weyl algebra \cite{[BO]} and some other
infinite dimensional Lie algebras \cite{[T], [Wa]}.
Specially, in \cite{[B]}, the author proved that all
the simple $\mathfrak{sl}_2$-modules fall into three families: highest (lowest) weight modules,
Whittaker modules and a third family obtained by localization;
in \cite{[C]}, the author studied a new class of modules,
which was called imaginary Whittaker module
over the non-twisted affine Lie algebra induced by
its parabolic subalgebras.
Inspired by recent activities on Whittaker modules
over various algebras, the authors in \cite{[BM]} described some general framework for
the study of Lie algebra modules locally finite over a subalgebra. In \cite{[MZ]},
in order to classify the simple modules over a family of finite dimensional
solvable Lie algebras, the authors studied and
classified the simple Virasoro modules which are locally finite over a positive part.

The Schr\"{o}dinger algebra is the Lie algebra of
Schr\"{o}dinger group, which is the symmetry
group of the free particle Schr\"{o}dinger equation.
The Schr\"{o}dinger algebra plays an important role in
mathematical physics and its applications \cite{[DDM]}.
Let $\mS$ denote the Schr\"{o}dinger algebra. Then
$\mS=\mathrm{span}_{\C}\{e,h,f,p,q,z\}$ with the following Lie brackets:
\begin{eqnarray}
\label{Sch-def}
\begin{split}
&[h,e]=2e, &[&h,f]=-2f, &[&e,f]=h,\\
&[h,p]=p, &[&h,q]=-q, &[&p,q]=z,\\
&[e,q]=p, &[&p,f]=-q, &[&f,q]=0,\\
&[e,p]=0, &[&z,\mS]=0. &
\end{split}\end{eqnarray}
From this we see that $\mS$ contains two subalgebras: the Heisenberg
subalgebra $\mathcal{H}=\mathrm{span}_\C\{p,q,z\}$
and $\mathfrak{sl}_{2}=\mathrm{span}_\C\{e,h,f\}$.
The Schr\"{o}dinger algebra $\mS$ can be viewed as a
semidirect product $\mS=\mathcal{H}\rtimes\mathfrak{sl}_{2}$.
$\mS$ has a triangular decomposition
\begin{equation}
\label{tri-decom}
\mS=\mS_+ \oplus \mS_0 \oplus \mS_+,
\end{equation}
where $\mS_+=\mathrm{span}_\C\{e, p\}$,
$\mS_0=\mathrm{span}_\C\{h, z\}$ and $\mS_-=\mathrm{span}_\C\{f, q\}$.

The representation theory of Schr\"{o}dinger algebra $\mS$
has attracted many authors' attention. For example,
using the technique of singular vectors,
a classification of the irreducible lowest weight representations
of the Schr\"{o}dinger algebra is given in \cite{[DDM]};
a classification of simple weight modules with finite dimensional weight spaces in \cite{[D]};
And a classification of finite-dimensional
indecomposable modules in \cite{[Wu], [WZ]}.
All the simple weight modules were classified in \cite{[LMZ2]}.
The irreducible representations of conformal
Galilei algebras related with Schr\"{o}dinger algebra in $l$-spatial dimension
was studied in \cite{[LMZ1]}.
In special, in \cite{[ZC]}, the authors studied the Whittaker modules
induced by the triangular decomposition (\ref{tri-decom}) of $\mS$, simple Whittaker modules and
related Whittaker vectors were determined.

In order to classify the simple modules over a family of finite dimensional
solvable Lie algebras, in this paper, we generalize the
Whittaker module for  Schr\"{o}dinger algebra $\mS$. Different from the
method of Ref. \cite{[ZC]}, this type of modules
are neither induced by its Borel subalgebra, nor by its parabolic subalgebra,
but its Heisenberg subalgebra $\mH$.
We call this kind of new modules as the quasi-Whittaker modules.
We prove that, for a simple $\mS$-module,
$V$ is a quasi-Whittaker module if and only if $V$ is a locally finite $\mH$-module.
Using the elements with the action
 similar to the center elements in $U(\mS)$
and their quasi-Whittaker vectors, we classify the simple
quasi-Whittaker modules for Schr\"{o}dinger algebra $\mS$.
Without the help of triangular decomposition of the
Schr\"{o}dinger algebra, the methods and technique of
dealing with the quasi-Whittaker module studied in this paper
will be interesting.

The paper is organized as follows. In section 2, we define the quasi-Whittaker module induced by its Heisenberg subalgebra $\mH$, give
some notations and formulaes and show that any simple
quasi-Whittaker module is equivalence to simple
the locally finite $\mH$-module. In section 3, we characterize the
quasi-Whittaker vectors of the quasi-Whittaker module.
In section 4, we
classify the simple quasi-Whittaker modules
for the Schr\"{o}dinger algebra;
In section 5, we describe arbitrary
quasi-Whittaker modules with generating quasi-Whittaker vectors.

\vs{16pt}
\par
\cl{\bf\S2. \ Preliminaries }
\setcounter{section}{2}\setcounter{theo}{0}\setcounter{equation}{0}
\par
First of all, we give the basic definitions for this paper.
\begin{defi}
\label{def-00}
Let $\phi:\mathcal{H}\rightarrow\mathbb{C}$ be any Lie
algebra homomorphism. Let $V$ be a $\mathfrak{S}$-module.

$\mathrm{(i)}$ A nonzero vector $v\in V$ is called a quasi-Whittaker vector of type
$\phi$ if $xv=\phi(x)v$ for all $x\in\mathcal{H}.$

$\mathrm{(ii)}$ $V$ is called a quasi-Whittaker module for $\mathfrak{S}$ of type
$\phi$ if $V$ contains a cyclic quasi-Whittaker vector $v$ of type $\phi$.
\end{defi}
\begin{rema}
From the definition, it is easy to see that $\phi(z)=0$.
\end{rema}

\begin{lemm}
\label{zero}
If $\phi$ is a zero homomorphism, then $V$ is a simple
quasi-Whittaker module of type $\phi$ if and only if
it is a simple $\mathfrak{sl}_{2}$-module.
\end{lemm}
\begin{proof}
By definition, the necessary condition is clear,
thus we just prove the sufficient condition.

Let $w$ be the cyclic quasi-Whittaker vector of $V$,
then $pw=qw=0$. Let $W=\{v\in V|pv=qv=0\}$. Since
$$p(e^{i}h^{j}f^{k})=q(e^{i}h^{j}f^{k})=0.$$
$W$ is a submodule of $V$. Moreover, the actions of $p$
and $q$ commutes, thus $W$ is nonzero. By the simplicity of $V$,
we know $W=V$. Therefore $pV=qV=0$. It is clear that $zV=0$.
This implies that $V$ is a simple $\mathfrak{sl}_{2}-$module.
\end{proof}

\begin{defi}
\label{quasi-Whittaker}
Let $\phi:\mH\rightarrow\mathbb{C}$ be any nonzero Lie
algebra homomorphism. Define a one-dimensional
$\mH$-module $\mathbb{C}_{\phi}=\mathbb{C}w$ by
$pw=\phi(p)w, qw=\phi(q)w.$ Then we get an induced module
\begin{equation}
\label{Whitt-mod}
M_{\phi}=\mathcal {U}(\mathfrak{S})
\otimes_{\mathcal {U}(\mH)}\mathbb{C}_{\phi},
\end{equation}
which is called the universal quasi-Whittaker module of type $\phi$ in the
sense that for any quasi-Whittaker module $V$ with quasi-Whittaker vector $v$ of
type $\phi,$ there is a unique surjective homomorphism
$f:M_{\phi}\rightarrow V$ such that $uw\mapsto uv, \forall
u\in\mathcal {U}(\mathfrak{sl_{2}}).$
\end{defi}

The following notations will be used to describe
bases for $\U(\mS)$ and for quasi-Whittaker modules.
Fix $\phi: \mH\rightarrow\C$ be any nonzero Lie
algebra homomorphism, define the following elements in $\U(\mS)$:
\begin{align*}
X&=\delta_{\phi(q),0}e+\delta_{\phi(p),0}f,\\
C&=\phi(p)^{2}f-\phi(q)^{2}e-\phi(p)\phi(q)h,\\
P_{+}&=\delta_{\phi(p),0}p+\delta_{\phi(q),0}q,\\
P_{-}&=\delta_{\phi(q),0}p+\delta_{\phi(p),0}q.
\end{align*}

\begin{rema}
It is easy to see that $X,h,C$ forms a basis
for $\mathfrak{sl}_{2}$ when $\phi$ is nonzero with $\phi(p)\phi(q)=0$.
\end{rema}

\begin{defi}
Let $M_{\phi}$ be the universal quasi-Whittaker
module with cyclic quasi-Whittaker vector $w$.
For $\xi\in\C$, define a submodule $W_{\phi,\xi}$ as
$$W_{\phi,\xi}=\U(\mS)(C-\xi)w.$$
And let $$L_{\phi,\xi}=M_{\phi}/W_{\phi,\xi}.$$
\end{defi}
Denote "$\bar{\quad}$" the canonical projection
from $M_{\phi}$ to $L_{\phi,\xi}$.

Recall that a locally finite module for a Lie algebra is defined as follow.
\begin{defi}
Let $L$ be a Lie algebra. An $L-$module is called locally finite
if any nonzero element $v$ in $V$ is contained in a finite dimensional submodule.
\end{defi}

From the following lemma, we can identify any simple $\mS-$module
on which $\mH$ act locally finite as a quasi-Whittaker module.
\begin{theo}
\label{loc-fin}
Let $V$ be a simple $\mS$-module, then the following two conditions are equivalent.

$\mathrm{(i)}$ $V$ is a locally finite $\mH$-module.

$\mathrm{(ii)}$ $V$ is a quasi-Whittaker module for $\mS$.
\end{theo}
\begin{proof}
$\mathrm{(ii)\Rightarrow(i)}$. This is clear by definitions.

$\mathrm{(i)\Rightarrow(ii)}$. Take any nonzero $v_{0}\in V$, then there exists a submodule $V_{1}\subseteq V$ such that $V_{1}$ is a finite dimensional $\mH$-module containing $v_{0}$. Note that $\mH$ is solvable, following from Lie's theorem, $p$ and $q$ have common eigenvector $v$ in $V_{1}$, and $zv=[p,q]v=0$. Since $V$ is simple as $\mS$-module, we see that $V=\U(\mS)v$ and $zV=0$. By definition, $V$ is a quasi-Whittaker module for $\mS$.
\end{proof}

By (\ref{Sch-def}), we can deduce by induction
the following identities which will be used later.
\begin{lemm} For any $m, n, k\in\mathbb{N}$, we have
\begin{align}
\label{rel1}
p^{n}h^{m}&=\sum\limits_{i=0}^{m}\binom{m}{i}(-1)^{i}n^{i}h^{m-i}p^{n}, \\
q^{n}h^{m}&=\sum\limits_{i=0}^{m}\binom{m}{i}n^{i}h^{m-i}q^{n}, \\
pf^{k}&=f^{k}p-kf^{k-1}q, \\
\label{rel2}qe^{k}&=e^{k}q-ke^{k-1}p.
\end{align}
\end{lemm}
\vs{6pt}
\par
\cl{\bf\S3. \ Quasi-Whittaker Vectors in $M_{\phi}$ and $L_{\phi,\xi}$}
\setcounter{section}{3}\setcounter{theo}{0}\setcounter{equation}{0}
\par

In the following sections, we will discuss the reducibility of the quasi-Whittaker
modules for the Schr\"{o}dinger algebra. There are four
conditions according to the choice of $\phi$ in Definition \ref{def-00},
that is, $\phi(p)=\phi(q)=0$; $\phi(p)=0$ while $\phi(q)\neq0$;
$\phi(p)\neq0$ while $\phi(q)=0$; $\phi(p)\neq0$ and $\phi(q)\neq0$.
The simplest condition of the four cases is $\phi(p)=\phi(q)=0$.
In this case, from Lemma \ref{zero}, any a $\mathfrak{sl}_2$-module is
a quasi-Whittaker module. While the
reducibility of $\mathfrak{sl}_2$-module was completely
determined by R. Block in \cite{[B], [M]}. Thus,
in this paper, we do not deal with this condition.

Next we will
discuss the quasi-Whittaker module for $\mS$ according to the remaining three
conditions.
When $\phi(p)\phi(q)=0$, we take a basis
for $M_{\phi}$ as $\{X^{i}h^{j}C^{k}w|i,j,k\in\mathbb{Z}_{+}\}$,
while, when $\phi(p)\phi(q)\neq0$, we take it
as $\{h^{i}f^{j}C^{k}w\}$, where $w$ is a cyclic quasi-Whittaker vector of $M_{\phi}$.
\begin{lemm}
\label{Wh-vec}
For any $k\in\mathbb{Z}$, we have
\begin{align*}
(p-\phi(p))C^{k}w=0,\\
(q-\phi(q))C^{k}w=0.
\end{align*}
Thus, any vector in $\C[C]w$ is a quasi-Whittaker vector of type $\phi$.
\end{lemm}
\begin{proof}
The lemma follows from a direct computation.
\end{proof}

From (\ref{rel1})-(\ref{rel2}), we know that
\begin{lemm}
Suppose $\phi(p)\phi(q)=0$, then
\begin{align}
\label{p+}(P_{+}-\phi(P_{+}))X^{i}&=X^{i}P_{+}-iX^{i-1}P_{-},\\
\label{p-}(P_{-}-\phi(P_{-}))h^{m}&
=\sum\limits_{i=1}^{m}\binom{m}{i}(-1)^{(\delta_{\phi(p),0}+1)i}h^{m-i}P_{-}.
\end{align}
\end{lemm}

\begin{lemm}
\label{step01}
Assume $\phi(p)\phi(q)=0$. Suppose
$$x=\sum\limits_{i=0}^{n}X^{i}
\sum\limits_{j=0}^{m}h^{j}a_{ij}(C)w\in M_{\phi},a_{ij}(C)\in\mathbb{C}[C]. $$
Then
\begin{eqnarray}
\nonumber&\!\!\!\!\!\!\!\! &
(P_{+}-\phi(P_{+}))^{n}x\\[6pt]
\nonumber&\!\!\!\!\!\!\!\!
&=(-1)^{n}n!\sum\limits_{j=0}^{m}P_{-}^{n}h^{j}a_{nj}(C)w\\[6pt]
\nonumber&\!\!\!\!\!\!\!\!
&=(-1)^{n}n!(\phi(p)+\phi(q))^{n}\sum\limits_{j=0}^{m}h^{j}b_{j}(C)w.
\end{eqnarray}
\end{lemm}
\begin{proof}
By (\ref{rel1})-(\ref{rel2}), we see that
\begin{align}
\label{p++}(P_{+}-\phi(P_{+}))h^{k}C^{i}w=0, \\
\label{p--}(P_{-}-\phi(P_{-}))C^{k}w=0.
\end{align}
Hence,
\begin{align*}
(P_{+}-\phi(P_{+}))^{n}x&=(P_{+}-\phi(P_{+}))
^{n-1}\sum\limits_{i=0}^{n}P_{+}X^{i}
\sum\limits_{j=0}^{m}h^{j}a_{ij}(C)w\\
&=(P_{+}-\phi(P_{+}))^{n-1}\sum\limits_{i=0}
^{n}(X^{i}P^{+}-iX^{i-1}P_{-})
\sum\limits_{j=0}^{m}h^{j}a_{ij}(C)w\\
&=-(P_{+}-\phi(P_{+}))^{n-1}
\sum\limits_{i=1}^{n}iX^{i-1}P_{-}
\sum\limits_{j=0}^{m}h^{j}a_{ij}(C)w\\
&=(P_{+}-\phi(P_{+}))^{n-2}
\sum\limits_{i=1}^{n}i(i-1)X^{i-2}P_{-}^{2}
\sum\limits_{j=0}^{m}h^{j}a_{ij}(C)w\\
&=(-1)^{n}n!\sum\limits_{j=0}^{m}P_{-}^{n}h^{j}a_{nj}(C)w\\
&=(-1)^{n}n!\sum\limits_{j=0}^{m}
\sum\limits_{s=0}^{j}\binom{j}{s}h^{j-s}P_{-}^{n}a_{nj}(C)w\\
&=(-1)^{n}n!(\phi(p)+\phi(q))^{n}\sum\limits_{j=0}^{m}h^{j}b_{j}(C)w.
\end{align*}
\end{proof}

For later use, we also need the following lemmas.
\begin{lemm}
\label{step02}
Assume that $\phi(p)\phi(q)=0.$
$$(P_{-}-\phi(P_{-}))^{s}(h^{t}b(C)w)=\left\{\begin{array}{ll}
0, &\mbox{if}\,\ s>t;\\
(\phi(p)+\phi(q))^{t}(-1)^{(\delta_{\phi(p),0}+1)t}t!b(C)w, &\mbox{if}\,\ s=t.
\end{array}\right.$$
where $b(C)\in\mathbb{C}[C]$.
\end{lemm}
\begin{proof}
We prove the second identity firstly.
If $t=1$, we have
\begin{align*}
(P_{-}-\phi(P_{-}))^{2}(hb(C)w)&=(P_{-}
-\phi(P_{-}))((-1)^{\delta_{\phi(p),0}+1}P_{-}b(C)w)\\
&=(-1)^{\delta_{\phi(p),0}+1}\phi(P_{-})(P_{-}-\phi(P_{-}))(b(C)w)\\
&=0.
\end{align*}
Hence, $(P_{-}-\phi(P_{-}))^{s}(hb(C)w)=0, \forall s>1$.
Assume that the identity holds for $t\leq k$,
then for $t=k+1$, we have
\begin{align*}
&(P_{-}-\phi(P_{-}))^{k+2}(h^{k+1}b(C)w)\\
&=(P_{-}-\phi(P_{-}))
^{k+1}(\sum\limits_{i=1}^{k+1}\binom{k+1}{i}(-1)
^{(\delta_{\phi(p),0}+1)i}h^{k+1-i}P_{-}b(C)w)\\
&=(P_{-}-\phi(P_{-}))^{k+1}(\phi(P_{-})
\sum\limits_{i=1}^{k+1}\binom{k+1}{i}(-1)^{(\delta_{\phi(p),0}+1)i}h^{k+1-i}b(C)w)\\
&=0.
\end{align*}
Now we turn to prove the first identity by induction. It is easy to check that
$$(P_{-}-\phi(P_{-}))(hb(C)w)=P_{-}b(C)w=\phi(P_{-})b(C)w. $$
Suppose the identity holds for $t\leq k$, then for $t=k+1$, we have
\begin{align*}
&(P_{-}-\phi(P_{-}))^{k+1}(h^{k+1}b(C)w)\\
&=(P_{-}-\phi(P_{-}))^{k}(\sum\limits_{i=1}^{k+1}
\binom{k+1}{i}(-1)^{(\delta_{\phi(p),0}+1)i}h^{k+1-i}P_{-}b(C)w)\\
&=(P_{-}-\phi(P_{-}))^{k}(\phi(P_{-})\sum\limits_{i=1}^{k+1}
\binom{k+1}{i}(-1)^{(\delta_{\phi(p),0}+1)i}h^{k+1-i}b(C)w)\\
&=(P_{-}-\phi(P_{-}))^{k}((-1)
^{(\delta_{\phi(p),0}+1)}(k+1)\phi(P_{-})h^{k}b(C)w)\\
&=\phi(P_{-})^{k+1}(k+1)!(-1)^{(\delta_{\phi(p),0}+1)(k+1)}b(C)w\\
&=(\phi(p)+\phi(q))^{k+1}(k+1)!(-1)^{(\delta_{\phi(p),0}+1)(k+1)}b(C)w.
\end{align*}
\end{proof}

\begin{lemm}
\label{step11}
Assume that $\phi(p)\phi(q)\neq0$. Then
$$(q-\phi(q))^{s}(h^{t}\sum\limits_{i=0}^{m}
f^{i}a_{i}(C))=\left\{\begin{array}{ll}
0, &\mbox{if}\,\ s>t;\\
t!(\phi(q))^{t}\sum\limits_{i=0}^{m}f^{i}a_{i}(C)), &\mbox{if}\,\ s=t,
\end{array}\right.$$
where $a_{i}(C)\in\C[C],i=1,\cdots,m$.
\end{lemm}
\begin{proof}
It is easy to check that $(q-\phi(q))f^{j}a(C)w=0,
\forall j\in\mathbb{Z}_{+},a(C)\in\C[C]$.
We use the induction on $t$. If $t=1$, then
\begin{align*}
(q-\phi(q))(h\sum\limits_{i=0}^{m}f^{i}a_{i}(C)w)&
=q\sum\limits_{i=0}^{m}f^{i}a_{i}(C)w=\phi(q)
\sum\limits_{i=0}^{m}f^{i}a_{i}(C)w,\\
(q-\phi(q))^{2}(h\sum\limits_{i=0}^{m}f
^{i}a_{i}(C)w)&=(q-\phi(q))(\phi(q)
\sum\limits_{i=0}^{m}f^{i}a_{i}(C)w)=0.
\end{align*}
Thus the identity holds for $t=1$.
Assume that the identity holds for $n=k$, then for $n=k+1$, we have
\begin{align*}
(q-\phi(q))^{k+1}(h^{k+1}\sum\limits_{i=0}
^{m}f^{i}a_{i}(C)w)&=(q-\phi(q))^{k}(\sum\limits
_{j=1}^{k+1}\binom{k+1}{j}h^{k+1-j}q\sum\limits_{i=0}^{m}f^{i}a_{i}(C)w)\\
&=(q-\phi(q))^{k}(\phi(q)\sum\limits_
{j=1}^{k+1}\binom{k+1}{j}h^{k+1-j}\sum\limits_{i=0}^{m}f^{i}a_{i}(C)w)\\
&=(q-\phi(q))^{k}(\phi(q)(k+1)h^{k}\sum\limits_{i=0}^{m}f^{i}a_{i}(C)w)\\
&=(k+1)!(\phi(q))^{k+1}\sum\limits_{i=0}^{m}f^{i}a_{i}(C)w.
\end{align*}
\end{proof}

\begin{lemm}
\label{step12}
Assume that $\phi(p)\phi(q)\neq0$. Then
$$(p-\phi(p))^{m}(\sum\limits_{i=0}
^{m}f^{i}a_{i}(C)w)=(-1)^{m}m!(\phi(q))^{m}a_{m}(C)w.$$
\end{lemm}
\begin{proof}
\begin{align*}
(p-\phi(p))^{m}(\sum\limits_{i=0}^{m}f^{i}
a_{i}(C)w)&=(p-\phi(p))^{m-1}(-\sum\limits_{i=1}^{m}if^{i-1}qa_{i}(C)w)\\
&=(p-\phi(p))^{m-1}(-\phi(q)\sum\limits_{i=1}^{m}if^{i-1}a_{i}(C)w)\\
&=(-1)^{m}m!(\phi(q))^{m}a_{m}(C)w.
\end{align*}
\end{proof}

By now, we can determine the quasi-Whittaker vectors for $M_{\phi}$ and $L_{\phi,\xi}$.
\begin{prop}
\label{Wh-vec0}
Let $M_{\phi}$ be the universal quasi-Whittaker module
generated by quasi-Whittaker vector $w$. Suppose
$w'\in M_{\phi}$ is a quasi-Whittaker vector,
then $w'\in\mathbb{C}[C]w$.
\end{prop}
\begin{proof}
(1) Assume that $\phi(p)\phi(q)=0$. Suppose
$$w'=\sum\limits_{i=0}^{n}X^{i}\sum\limits_{j=0}
^{m}h^{j}a_{ij}(C)w$$
with $a_{ij}(C)\in\mathbb{C}[C]$.
Then by Lemma \ref{step01}, we see that if $(P_{+}-\phi(P_{+}))w'=0$,
then $n=0$, that is $w'=\sum\limits_{j=0}^{m}h^{j}a_{j}(C)w$.
Applying Lemma \ref{step02}, we get $m=0$. Thus,
we have $w'=b(C)w\in\mathbb{C}[C]w$.
From Lemma \ref{Wh-vec} we know the proposition holds.

(2) If $\phi(p)\phi(q)\neq0$, we may assume that
$$w'=\sum\limits_{i=0}^{m}h^{i}\sum\limits_{j=0}^{n}
f^{j}a_{ij}(C)w$$
with $a_{ij}(C)\in\mathbb{C}[C]$.
Then by Lemma \ref{step11}, we see that if $(q-\phi(q))w'=0$,
then $m=0$, that is $w'=\sum\limits_{j=0}^{n}f^{j}a_{j}(C)w$.
Applying Lemma \ref{step12}, we get $n=0$. Thus, we have $w'=b(C)w\in\mathbb{C}[C]w$.
From Lemma \ref{Wh-vec} we know the proposition holds.
\end{proof}

\begin{prop}
\label{Wh-vec1}
Let $w$ be the cyclic quasi-Whittaker vector of $M_{\phi}$,
$\bar{w}\in L_{\phi,\xi}$. If $w'\in L_{\phi,\xi}$ is a
quasi-Whittaker vector, then $w'=c\bar{w}$ for some $c\in\C$.
\end{prop}
\begin{proof}
We only show the statement for $\phi(p)\phi(q)=0$
since the proof is similar when $\phi(p)\phi(q)\neq0$.
Note that $\{X^{i}h^{j}\bar{w}|i,j\in\mathbb{Z}_{+}\}$
spans $L_{\phi,\xi}$. We claim that this set is linearly
independent and thus a basis for $L_{\phi,\xi}$.
To check this, suppose
$$0=\sum\limits_{i,j}a_{ij}X^{i}h^{j}\bar{w}
=\overline{\sum\limits_{i,j}a_{ij}X^{i}h^{j}w}. $$
Then $\sum\limits_{i,j}a_{ij}X^{i}h^{j}w\in\U(\mS)(C-\xi)w$, and so
$$\sum\limits_{i,j}a_{ij}X^{i}h^{j}w=\sum\limits_{i,j}
\sum\limits_{k=0}^{m}b_{ij}^{k}X^{i}h^{j}C^{k}(C-\xi)w$$
for some $m\in\mathbb{Z}_{>0}$ and $b_{ij}^{k}\in\C$.
This expression can be rewritten as
$$\sum\limits_{i,j}(a_{ij}+\xi b_{ij}^{0})X^{i}h^{j}w
+\sum\limits_{i,j}\sum\limits_{k=1}^{m}
(\xi b_{ij}^{k}-b_{ij}^{k-1})X^{i}h^{j}C^{k}w
-\sum\limits_{i,j}b_{ij}^{m}X^{i}h^{j}C^{m+1}w=0. $$
From this we conclude that $b_{ij}^{m}=0,\xi b_{ij}^{k}
-b_{ij}^{k-1}=0,a_{ij}+\xi b_{ij}^{0}=0$, and thus $a_{ij}=0$ for all $i,j$.
With this fact now established, it is possible to
use the same argument as in Proposition 3.7 to complete the proof.
\end{proof}

\vs{6pt}
\par
\cl{\bf\S4. \ Simple Quasi-Whittaker modules for $\mathfrak{S}$}
\setcounter{section}{4}\setcounter{theo}{0}\setcounter{equation}{0}
\par

In this section we will determine all simple
quasi-Whittaker modules of type $\phi$, up to isomorphism. After this,
following from Theorem \ref{loc-fin}, we can classify all simple
$\mS$-modules on which $\mH$ acts locally finite.
\begin{theo}
\label{contain}
Let $V$ be a quasi-Whittaker module for $\mS$, and
let $W\subseteq V$ be a nonzero submodule.
Then there is a nonzero quasi-Whittaker vector $w'\in W$.
\end{theo}
\begin{proof}
\textbf{Case1.} If $\phi(p)\phi(q)=0$, let
$$x=\sum\limits_{i=0}^{n}X^{i}\sum\limits_{j=0}^{m}h^{j}a_{ij}(C)w\in W. $$
Then by Lemma \ref{step01}, we see that $\sum\limits_{j=0}^{m}h^{j}b_{j}(C)w\in W$.
Applying Lemma \ref{step02}, we deduce that $W$ contains an element $w'=b(C)w\neq0$.
From Proposition \ref{Wh-vec0}, we know that $w'$ is a quasi-Whittaker vector.

\textbf{Case2.} If $\phi(p)\phi(q)\neq0$,
let $$x=\sum\limits_{i=0}^{n}h^{i}
\sum\limits_{j=0}^{m}f^{j}a_{ij}(C)w\in W. $$
Then by Lemma \ref{step11}, we see that
$\sum\limits_{j=0}^{m}f^{j}b_{j}(C)w\in W$.
Applying Lemma \ref{step12}, we deduce that $W$
contains an element $w'=b(C)w\neq0$.
From Proposition \ref{Wh-vec0}, we know that $w'$ is a quasi-Whittaker vector.
\end{proof}

\begin{prop}
\label{max}
For any $\xi\in\mathbb{C}$, $W_{\phi,\xi}$ is
a maximal submodule of $M_{\phi}$, hence $L_{\phi,\xi}$ is simple.
\end{prop}
\begin{proof}
It is clear that $W_{\phi,\xi}$ is a proper submodule of $M_{\phi}$.\\
\textbf{Case 1.} If $\phi(p)\phi(q)=0$, then for any
$x\not\in W_{\phi,\xi}$, we see that
$$x\equiv\sum\limits_
{i,j}X^{i}h^{j}w(\mathrm{mod} W_{\phi,\xi}). $$
Let $x'
=\sum\limits_{i,j}X^{i}h^{j}w$, by Lemma \ref{step01} and Lemma
\ref{step02}, we can deduce that $w\in\mathcal{U}(\mS)x'+W_{\phi,\xi}$.
Hence, $W_{\phi,\xi}$ is a maximal submodule of $M_{\phi}$.\\
\textbf{Case 2.} If $\phi(p)\phi(q)\neq0$, then for any
$x\not\in W_{\phi,\xi}$, we see that
$$x\equiv\sum\limits_{i,j}h^{i}
f^{j}w(\mathrm{mod} W_{\phi,\xi}). $$
Let $x'=\sum\limits_{i,j}h^{i}f^{j}w$, by Lemma \ref{step11} and
Lemma \ref{step12}, we can deduce that $w\in\mathcal{U}(\mS)x'
+W_{\phi,\xi}$. Hence, $W_{\phi,\xi}$
is a maximal submodule of $M_{\phi}$.
\end{proof}

\begin{theo}
\label{simple-mod}
Let $\phi: \mH\rightarrow\mathbb{C}$ be a Lie algebra
homomorphism such that $\phi(p)\neq0$, or $\phi(q)\neq0$,
and let $V$ be a simple quasi-Whittaker module of type $\phi$ for
$\mS$. Then $V\cong L_{\phi,\xi}$ for some $\xi\in\mathbb{C}$.
\end{theo}
\begin{proof}
Let $w_{1}$ be a cyclic quasi-Whittaker vector corresponding to
$\phi$. From \cite{[LMZ2]}, we know that
$$[\mathcal{U}(\mS),p^{2}f-q^{2}e-hpq]V=0. $$
By Schur's lemma, we have $C_{0}=p^{2}f-q^{2}e-hpq$
acts on $V$ as a scalar $\xi$. Then $\xi w_{1}=C_{0}w_{1}=Cw_{1}$.
Thus $C$ acts on $w$ as by the scalar $\xi$. Now by the
universal property of $M_{\phi}$, there exists a
homomorphism $\varphi: M_{\phi}\rightarrow W$ with
$uw\mapsto uw_{1}$. This is a surjective map since $W$
is generated by $w_{1}$.
But $\varphi(W_{\phi,\xi})=\mathcal{U}(\mS)(C-\xi)w_{1}=0$, so we have
$$W_{\phi,\xi}\subseteq\mathrm{Ker}\varphi\subsetneq M_{\phi}$$
By Proposition \ref{max}, $W_{\phi,\xi}$ is maximal,
hence $\mathrm{Ker}\varphi=W_{\phi,\xi}$. That is $W\cong L_{\phi,\xi}$.
\end{proof}

Note that if $w$ is a quasi-Whittaker vector of type $\phi$, then we have $Cw=C_{0}w$. Hence, in the theorems and properties above, $C$ can be replaced by $C_{0}$. For the remaining part of this section, we will discuss the annihilator of quasi-Whittaker vectors for simple quasi-Whittaker modules.
\begin{lemm}
\label{ann1}
Fix $\phi: \mH\rightarrow\mathbb{C}$.
Define the left ideal $L$ of $\mathcal{U}(\mS)$
by $L=\mathcal{U}(\mS)(C_{0}-\xi1)+\mathcal{U}
(\mS)(p-\phi(p)1)+\mathcal{U}(\mS)(q-\phi(q)1)$,
and regard $V=\mathcal{U}(\mS)/L$ as a left $\mathcal{U}(\mS)-$module.
Then $V\cong L_{\phi,\xi}$, and thus $V$ is simple.
\end{lemm}
\begin{proof}
For $u\in\mathcal{U}(\mS)$, let $\bar{u}=u+L\in\mathcal{U}(\mS)/L$.
Then we may regard $V$ as a quasi-Whittaker module of type $\phi$ with
cyclic quasi-Whittaker vector $\bar{1}$. By the universal property of
$M_{\phi}$, there exists a homomorphism $\varphi: M_{\phi}\rightarrow V$
with $uw\mapsto u\bar{1}$. This is a surjective map since $\bar{1}$
is the generator of $V$. However, for any $u(C_{0}-\xi)w\in W_{\phi,\xi}$,
we have $\varphi(u(C_{0}-\xi)w)=u(C_{0}-\xi)\bar{1}=0$. Hence,
$$W_{\phi,\xi}\subseteq\mathrm{Ker}\varphi\subseteq M_{\phi}. $$
Since $W_{\phi,\xi}$ is maximal, it follows that
$V\cong M_{\phi}/\mathrm{Ker}\varphi\cong L_{\phi,\xi}$.
\end{proof}

\begin{prop}
\label{ann}
Let $V$ be a quasi-Whittaker module of type $\phi$ such
that $C_{0}$ acts on the cyclic quasi-Whittaker vector by the
scalar $\xi\in\mathbb{C}$. Then $V$ is simple. Moreover,
if $w$ is a cyclic quasi-Whittaker vector for $V$,
then $$\mathrm{Ann}_{\mathcal{U}(\mS)}(w)=\mathcal{U}(\mS)
(C_{0}-\xi1)+\mathcal{U}(\mS)(p-\phi(p)1)+\mathcal{U}(\mS)(q-\phi(q)1). $$
\end{prop}
\begin{proof}
Let $K$ denote the kernel of the natural surjective map
$\mathcal{U}(\mS)\rightarrow V$ given by $u\mapsto uw$.
Then $K$ is a proper left ideal containing
$$L=\mathcal{U}(\mS)(C_{0}-\xi1)+\mathcal{U}(\mS)(p-\phi(p)1)
+\mathcal{U}(\mS)(q-\phi(q)1). $$ By Lemma \ref{ann1},
$L$ is maximal, thus $K=L$ and $V\cong\mathcal{U}(\mS)/L$ is simple.
\end{proof}

\vs{6pt}
\par
\cl{\bf\S5. \ Arbitrary Quasi-Whittaker Modules }
\setcounter{section}{5}\setcounter{theo}{0}\setcounter{equation}{0}
\par
In this section, we always assume that $V$ is any
a quasi-Whittaker module of type $\phi$
with cyclic quasi-Whittaker vector $w$, similar to \cite{[OW]}, we will
describe its reducibility and its quasi-Whittaker vectors
in terms of $\mathrm{Ann}_{\C[C_{0}]}(w)$.
\begin{lemm}
\label{monic}
Assume that
$\mathrm{Ann}_{\C[C_{0}]}(w)=(C_{0}-\xi1)^{k}$ for some $k>0$,
and the submodules sequence
$$V=V_{0}\supseteq V_{1}\supseteq\cdots
\supseteq V_{k}=0$$
are defined by $V_{i}=\mathcal{U}(\mS)(C_{0}-\xi1)^{i}w$. Then
$\mathrm{(i)}$ $V_{i}$ is a quasi-Whittaker module of type $\phi$,
with cyclic quasi-Whittaker vector $w_{i}=(C_{0}-\xi1)^{i}w$
and $\mathrm{(ii)}$ $V_{i}/V_{i+1}$ is simple for $0\leq i<k$.
Particularly,
$\mathrm{(iii)}$ the submodules $V_{0},\cdots,V_{k}$ are the only submodules of $V$.
\end{lemm}
\begin{proof}
It is easy to see that $V_{i}$ is a quasi-Whittaker module
with cyclic quasi-Whittaker vector $w_{i}$ by (\ref{p-}) and (\ref{p--}).
Obviously, $V_{i}/V_{i+1}$ is a quasi-Whittaker module of type
$\varphi$ with quasi-Whittaker vector $\bar{w_{i}}$. Since $C_{0}$
acts by the scalar on $\bar{w_{i}}$, following from
Proposition \ref{ann}, we deduce that $V_{i}/V_{i+1}$ is simple,
and thus isomorphic to $L_{\phi,\xi}$. Thus $V_{0}
\supseteq V_{1}\supseteq\cdots\supseteq V_{k}$ form a
composition series for $V$, and any simple subquotient
of $V$ is isomorphic to $V/V_{1}\cong L_{\phi,\xi}$.

Let $M$ be any maximal submodule of $V$, then $V/M$ is
a simple quasi-Whittaker module of type $\phi$ with quasi-Whittaker
vector $\bar{w}$. By Theorem \ref{simple-mod}, we see that $C_{0}$ acts
on $w$ by some scalar $\kappa\in\mathbb{C}$.
While, $(C_{0}-\xi1)^{k}$ acts as $0$ on $w$,
and therefore on $\bar{w}$. So we have $(\kappa-\xi)^{k}w\in M$.
Since $w\not\in M$, we deduce that $\kappa=\xi$.
It follows that $(C_{0}-\xi1)w\in M$,
thus $V_{1}=\mathcal{U}(\mS)(C_{0}-\xi1)w\subseteq M$.
Since $V_{i}$ is a maximal submodule of $M$,
we get $V_{1}=M$. Similarly, we can show that $V_{i+1}$
is the unique maximal submodule of $V_{i}$ for every $i<k$.
Thus $V_{i}$ for $1\leq i\leq k$ are the only submodules of $V$.
\end{proof}

\begin{theo}
\label{poly}
Assume that $\mathrm{Ann}_{\mathbb{C}[C_{0}]}(w)\neq0$ and
$d(C_{0})=\prod\limits_{i=1}^{k}(C_{0}-\xi_{i}1)^{a_{i}}$ for
distinct $\xi_{1},\cdots,\xi_{k}\in\mathbb{C}$ be
the unique monic generator of
$\mathrm{Ann}_{\mathbb{C}[C_{0}]}(w)$ in $\mathbb{C}[C_{0}]$,

$\mathrm{(i)}$ Define $V_{\phi,\xi_{i}}=\mathcal{U}(\mS)(C_{0}-\xi_{i})w,i=1,\cdots,k$, then $V_{\phi,\xi_{1}},\cdots,V_{\phi,\xi_{k}}$
are the only maximal submodules of $V$.

$\mathrm{(ii)}$ Define
$$w_{j}=d_{j}(C_{0})w, \mbox{where}\quad d_{j}(C_{0})=\prod\limits_{i\neq j}(C_{0}-\xi_{i}1)^{a_{i}},\mbox{and}\quad V_{j}=\U(\mS)w_{j}.$$
Then $V_{i}$ is a quasi-Whittaker module of type $\phi$ with cyclic quasi-Whittaker vector $w_{i}$ and $V=V_{1}\oplus\cdots\oplus V_{k}$. Furthermore, the submodules $V_{1},\cdots,V_{k}$ are indecomposable; $V_{j}$ is simple if and only if $a_{j}=1$; and $a_{j}$ is the composition length of $V_{j}$. In particular, $V$ has a composition series of length $\sum\limits_{i=1}^{k}a_{i}$.
\end{theo}
\begin{proof}
(i) Let $M$ be a maximal submodule of $V$, then $V/M$ is
simple and thus $C$ acts on $w$ by some scalar
$\kappa\in\mathbb{C}$. On the other hand, $d(C)$ acts as 0 on $w$,
and therefore on $\bar{w}$. Thus $d(\kappa)=0$,
which implies that $\kappa=\xi$ for some $1\leq i\leq k$.
Hence,
$$V_{\phi,\xi_{i}}=\mathcal{U}(\mS)(C-\xi_{i})w\subseteq M. $$
Since $V_{\phi,\xi_{i}}$ is the image of the maximal
submodule $W_{\phi,\xi_{i}}$ under the
epimorphism $M_{\phi}\rightarrow V$,
$V_{\phi,\xi_{i}}$ is maximal. Thus, $M=V_{\phi,\xi_{i}}$.

(ii)Firstly, we show that: $V=V_{1}+\cdots+ V_{k}$. Since $\mathrm{gcd}(d_{1},\cdots,d_{k})=1$, there exist polynomials $r_{1}(C_{0}),\cdots,r_{k}(C_{0})\in\C[C_{0}]$ such that $\sum\limits_{i=1}^{k}r_{i}(C_{0})d_{i}(C_{0})=1$. Therefore, $$w=1w=(\sum\limits_{i=1}^{k}r_{i}(C_{0})d_{i}(C_{0}))w\in V_{1}+\cdots+V_{k}.$$
To show that the sum $V=V_{1}+\cdots+V_{k}$ is direct, note that for $i\neq j, d(C_{0})$ is a factor of $d_{i}(C_{0})d_{j}(C_{0})$, which implies that $d_{j}(C_{0})w_{i}=0$. Following from this, we have
\begin{align*}
w_{i}&=1w_{i}\\
&=(r_{1}(C_{0})d_{1}(C_{0})+\cdots+r_{k}(C_{0})d_{k}(C_{0}))w_{i}\\
&=r_{i}(C_{0})d_{i}(C_{0})w_{i}.
\end{align*}
Suppose that $u_{1}w_{1}+\cdots+u_{k}w_{k}=0$
for $u_{1},\cdots,u_{k}\in\U(\mS)$, then
$$0=r_{i}(C_{0})d_{i}(C_{0})(\sum\limits_{j=1}^
{k}u_{j}w_{j})=u_{i}r_{i}(C_{0})d_{i}(C_{0})w_{i}=u_{i}w_{i}.$$
Thus, the sum is direct.

To finish the proof, by Lemma \ref{monic}, we know that the
submodules $V_{1},\cdots,V_{k}$ are indecomposable
with the stated composition length.
\end{proof}

If we consider the annihilator of $w$ in $\C[C]$,
then we can get similar result as follow:

\textbf{Theorem $\ref{poly}^{'}$} Assume that
$\mathrm{Ann}_{\mathbb{C}[C]}(w)\neq0$ and
$d(C)=\prod\limits_{i=1}^{k}(C-\xi_{i}1)
 ^{a_{i}}$ for distinct $\xi_{1},\cdots,\xi_{k}\in\mathbb{C}$
 be the unique monic generator of the ideal
$\mathrm{Ann}_{\mathbb{C}[C]}(w)$ in $ \mathbb{C}[C]$.

$\mathrm{(i)}$ Define $V_{\phi,\xi_{i}}=\mathcal{U}(\mS)(C-\xi_{i})w,i=1,\cdots,k$, then $V_{\phi,\xi_{1}},\cdots,V_{\phi,\xi_{k}}$ are the only maximal submodules of $V$.

$\mathrm{(ii)}$ Let $n=\deg d(C)$, then $V$ has a
unique composition series (up to permutation):
$V=V_{0}\supseteq V_{1}\supseteq \cdots\supseteq V_{n}=0$
with $V_{i}/V_{i+1}\cong L_{\phi,\xi_{j}}$ for some $j=1,\cdots,k$.
And the composition factors are $a_{i}$ copies
of $L_{\phi,\xi_{i}},i=1,\cdots,k$.$\hfill{}\Box$

However, we can not decompose $V$ as a direct
sum of submodules using $\mathrm{Ann}_{\C[C]}(w)$.

\begin{coro}
Assume that
$\mathrm{Ann}_{\mathbb{C}[C_{0}]}(w)\neq0$ and
$d(C_{0})$ be the unique monic generator of
$\mathrm{Ann}_{\mathbb{C}[C_{0}]}(w)$. Then
$$\mathrm{Ann}_{\mathcal{U}(\mS)}(w)
=\mathcal{U}(\mS)d(C_{0})+\U(\mS)(p-\phi(p)1)+\U(\mS)(q-\phi(q)1). $$
\end{coro}
\begin{proof}
We use induction on the composition length $n$ of $V$ (or
equivalently, the degree of $d(C_{0})$).
If $n=1$, then $d(C_{0})=C_{0}-\xi$, thus $V$ is simple,
therefore the result is true by Proposition \ref{ann}.
Assume that $n>1$, write $d(C_{0})=(C_{0}-\xi)d'(C_{0})$ for some $\xi\in\C$ and
some monic polynomial $d'(C_{0})\in\C[C_{0}]$ with positive degree. Then $w'=(C_{0}-\xi)w\neq0$.

Let $V'=\U(\mS)w'\subseteq V$. Then $V'$ is a
quasi-Whittaker module with cyclic quasi-Whittaker vector $w'$,
and $\mathrm{Ann}_{\C[C_{0}]}(w')=\C[C_{0}]d'(C_{0})$. Theorem \ref{poly}
therefore implies that the composition length of $V'$ is
$n-1$, and the induction hypothesis
deduce that
$$\mathrm{Ann}_{\C[C_{0}]}(w')=\U(\mS)d'(C_{0})+\U(\mS)
(p-\phi(p)1)+\U(\mS)(q-\phi(q)1). $$
Let $\bar{w}=w+V'\in V/V'$,
and note that $\mathrm{Ann}_{\C[C_{0}]}(\bar{w})=\C[C_{0}](C_{0}-\xi)$.

Let $u\in\mathrm{Ann}_{\U(\mS)}(w)$.
Since $\mathrm{Ann}_{\U(\mS)}(w)\subseteq\mathrm{Ann}_
{\U(\mS)}(\bar{w})$, following from proposition \ref{ann}, we have
\begin{align}
\nonumber u&=u_{0}(C_{0}-\xi1)+u_{1}(p-\phi(p))+u_{2}(q-\phi(q))\\
\label{in} &\in\U(\mS)(C-\xi1)+\U(\mS)(p-\phi(p))+\U(\mS)(q-\phi(q)).
\end{align}
While $u_{1}(p-\phi(p))+u_{2}(q-\phi(q))\in\mathrm{Ann}_
{\U(\mS)}(w)$, thus $u_{0}(C_{0}-\xi1)\in\mathrm{Ann}_{\U(\mS)}(w)$.
Observe that $0=u_{0}(C_{0}-\xi1)w=u_{0}w'$, we have
$$u_{0}\in\mathrm{Ann}_{\U(\mS)}(w')=\U(\mS)d'(C_{0})+
\U(\mS)(p-\phi(p)1)+\U(\mS)(q-\phi(q)1). $$
(\ref{in}) implies that $u$ has the required form.
\end{proof}

\begin{theo}
\label{submod}
Let $M_{\phi}$ be the universal quasi-Whittaker module of
type $\phi$ with cyclic quasi-Whittaker vector $w$.
If $U\subseteq M_{\phi}$ is a submodule, then
$U\cong M_{\phi}$. Furthermore, $V$ is generated by
a quasi-Whittaker vector of the form $d(C)w$ for some $d(x)\in\C[x]$.
\end{theo}
To prove this theorem, we need the following lemma.

\begin{lemm}
\label{prelem}
If $\mathrm{Ann}_{\C[C_{0}]}(w)=0$,
then $V\cong M_{\phi}$.
\end{lemm}
\begin{proof}
By the universal property of $M_{\phi}$, there exists
an epimorphism $\varphi: M_{\phi}\rightarrow V$.
It is clear that $\mathrm{Ker}\varphi$ is a submodule of $M_{\phi}$.
If $\mathrm{Ker}\varphi\neq0$, then Theorem \ref{contain} implies that there is
a nonzero quasi-Whittaker vector $w'\in \mathrm{Ker}\varphi$. It follows
from Proposition \ref{Wh-vec0}, $0\neq w'=d(C_{0})1\otimes1$ and
thus $0\neq d(C_{0})\in\mathrm{Ann}_{\C[C_{0}]}(w)$, which
is impossible. Therefore, $\mathrm{Ker}\varphi$ must be $0$ and $\varphi$ is an isomorphism.
\end{proof}

\textbf{Proof of Theorem \ref{submod}.}
Since $W_{\phi,\xi}$ is a maximal submodule
and a quasi-Whittaker module of type $\phi$ with
cyclic quasi-Whittaker vector $(C_{0}-\xi)w$,
using Lemma \ref{prelem}, we obtain that $W_{\phi,\xi}\cong M_{\phi}$.

By Proposition \ref{Wh-vec0} and Theorem \ref{contain}, we know that $V$
contains a quasi-Whittaker vector $d(C_{0})w$. By Lemma \ref{prelem},
we have $\U(\mS)d(C_{0})w\cong M_{\phi}$. Since $d(C_{0})$ can be written as a
product of linear factors, from the above we know that
there exists a chain of universal
quasi-Whittaker modules between $\U(\mS)d(C_{0})w$ and
$M_{\phi}$ satisfying that each quotient is simple.
Thus $V$ must be one of the submodules in the chain,
thus $V$ is a universal quasi-Whittaker module of type $\phi$.
$\hfill{}\Box$

\begin{theo}
The set of
quasi-Whittaker vectors in $V$ is $\C[C_{0}]w$.
\end{theo}
\begin{proof}
If $\mathrm{Ann}_{\C[C_{0}]}(w)=0$, using Proposition \ref{Wh-vec0} and Lemma \ref{prelem} we
know that the conclusion holds. If
$\mathrm{Ann}_{\C[C_{0}]}(w)\neq0$, by
Theorem \ref{poly}, we obtain that $V$ has finite composition
length $n$.

Next we induce on $n$.
If $n=1$, then $V$ is simple, using
Proposition \ref{Wh-vec1} we obtain the conclusion. Suppose that $V$
is a module with arbitrary composition length $n$.
Let $$V=V_{0}\supseteq V_{1}\supseteq\cdots\supseteq V_{n}=0$$
be a composition series of $V$, and assume $V_{1}$ has cyclic
quasi-Whittaker vector $w_{1}=(C_{0}-\xi1)w$. Let $w'\in V$ be a
quasi-Whittaker vector. Since $V/V_{1}$ is simple,
by Proposition \ref{Wh-vec1} we obtain that the image of $w'$
in $V/V_{1}$ is a scalar multiple of $\bar{w}$.
Therefore, in $V$, $w'=cw+w''$ for some $c\in\C$ and
$w''\in V_{1}$. Note that $w''=w'-cw$ is also a
quasi-Whittaker vector. Since $V_{1}$ has composition
length $n-1$, by induction, we have
$$w''=d(C_{0})w_{1}
=d(C_{0})(C_{0}-\xi1)w \ \mbox{for some} \ d(C_{0})\in\C[C_{0}]. $$
Therefore, $w'=cw+d(C_{0})(C_{0}-\xi1)w$, the statement holds.
\end{proof}
\vs{10pt}
\par


 \vs{10pt}
\par
\cl{\bf REFERENCES} \vs{-1pt}\small \lineskip=4pt


\begin{enumerate}
{\footnotesize
\bibitem{[BM]} P. Batra, V. Mazorchuk, Blocks and modules for Whittake rpairs,
{\it J. Pure Appl. Algebra} {\bf215} (2011), 1552-1568.

\bibitem{[BO]} G. Benkart, M. Ondrus, Whittaker modules for generalized Weyl algebras,
{\it Represent. Theory} {\bf 13} (2009), 141-164.

\bibitem{[B]} R. Block, The irreducible representations of the Lie algebra $sl(2)$ and of the
Weyl algebra. {\it Adv. Math.} {\bf39} (1981) 69-110.

\bibitem{[C]} K. Christodoupoulou, Whittaker modules for Heisenberg algebras
and imaginary Whittaker modules for affine Lie algebras,
{\it J. Algebra} {\bf 320} (2008), 2871-2890.

\bibitem{[CG]} H. Chen, X. Guo, New simple modules for the Heisenberg-Virasoro algebra,
{\it J. Algebra}, {\bf390} (2013), 77-86.

\bibitem{[D]} B. Dubsky, Classification of simple weight modules with
finite-dimensional weight spaces over the Schr\"{o}dinger
algebra, arXiv: 1309.1346v1, [math. RT].

\bibitem{[DDM]} V. Dobrev, H. Doebner, C. Mrugalla,
Lowest weight representations of the Schr\"{o}dinger algebra and generalized
heat Schr\"{o}dinger equations, {\it Rep. Math. Phys.} {\bf39} (1997), 201-218.

\bibitem{[GL]} X. Guo, X. Liu, Whittaker modules over generalized Virasoro algebra,
{\it J. Math. Phys.} {\bf 52}(9) (2011), 093504.

\bibitem{[GZ]} X. Guo, K. Zhao, Irreducible representations of
untwisted affine Kac-Moody algebras,
 arXiv:1305.4059, [math. RT].

\bibitem{[K]} B. Kostant, On Whittaker vectors and representation theory,
{\it Invent. Math.} {\bf48} (2) (1978), 101-184.

\bibitem{[LGZ]} R. L\"{u}, X. Guo, K. Zhao, Irreducible modules over the Virasoro algebra,
{\it Doc. Math.} {\bf 16} (2011), 709-721.

\bibitem{[LMZ1]} R. L\"{u}, V. Mazorchuk, K. Zhao, On simple modules over
conformal Galilei algebras,  arXiv:1310.6284 [math. RT].

\bibitem{[LMZ2]} R. L\"{u}, V. Mazorchuk, K. Zhao,
Classification of simple weight modules over the
Schr\"{o}dinger algebra, preprint, 2013.

\bibitem{[LWZ]} D. Liu, Y. Wu, L. Zhu, Whittaker modules for the twisted
Heisenberg-Virasoro algebra, {\it J. Math. Phys.} {\bf51} (2010), 023524,12pp.

\bibitem{[M]} V. Mazorchuk, Lectures on $sl_2(\C)$-modules, Imperial College, 2010.

\bibitem{[MZ]} V. Mazorchuk, K. Zhao,
Simple Virasoro modules which are locally finite over a positive part,
arXiv:1205.5937v1, {\it Sel. Math. New Ser.}
DOI 10.1007/s00029-013-0140-8 {September 2013}, to appear.

\bibitem{[O]} M. Ondrus, Whittaker modules for $U_q(sl_2)$,
{\it J. Algebra} {\bf 289} (2005), 192-213.

\bibitem{[OW]} M. Ondrus, E. Wiesner, Whittaker modules for the Virasoro algebra,
{\it J. Algebra Appl.} {\bf 8} (2009), 363-377.

\bibitem{[S]} A. Sevostyanov, Quantum deformation of Whittaker modules and the Toda lattice,
{\it Duke Math. J.} {\bf 105} (2) (2000), 211-238.

\bibitem{[T]} X. Tang, On Whittaker modules over a class of algebras similar to $U(sl_2)$,
{\it Front. Math. China} {\bf2} (2007), 127-142.

\bibitem{[Wa]} B. Wang, Whittaker modules for graded Lie algebras,
{\it Algebr. Represent. Theory} {\bf14} (2011), 691-702.


\bibitem{[Wu]} Y. Wu, Finite dimensional indecomposable modules for Schr\"{o}dinger algebra
 {\it J. Math. Phys.} {\bf54} (2013), 07350380 9pp.

\bibitem{[WZ]} Y. Wu, L. Zhu, Simple weight modules of
Schrodinger algebra, {\it Linear Algebr. Appl.} {\bf 438} (2013), 559-563.

\bibitem{[ZC]}
X. Zhang, Y. Cheng, Simple Schr\"{o}dinger modules
which are locally finite over the positive part,
arXiv: 1311.2118 [math. RT].

\bibitem{[ZTL]}
X. Zhang, S. Tan, H. Lian, Whittaker modules for the Schr\"{o}dinger-Witt algebra,
{\it J. Math. Phys.} {\bf 51} (2010), 083524, 17pp.

}
\end{enumerate}
\end{document}